\documentclass[12pt,reqno]{amsart}

\usepackage{fullpage}
\usepackage{amsmath,amsthm,amssymb}
\numberwithin{equation}{section}
\usepackage{fullpage}

\newcommand{\beq}{\begin{equation}}
\newcommand{\eeq}{\end{equation}}

\newcommand{\RR}{\ensuremath{\mathbb{R}}}
\newcommand{\tr}{\ensuremath{\top}}

\newcommand{\bpm}{\begin{pmatrix}}
\newcommand{\epm}{\end{pmatrix}}

\newcommand{\dif}{\ensuremath{\mathrm{d}}}
\newcommand{\mi}{\ensuremath{\mathrm{i}}}
\newcommand{\f}{\ensuremath{\mathsf{f}}}

\newcommand{\bfa}{\ensuremath{\mathbf{a}}}

\DeclareMathOperator{\re}{Re}

\renewcommand{\Re}{\re}

\newcommand{\lam}{\ensuremath{\lambda}}
\newcommand{\mat}[1]{\ensuremath{\mathbf{#1}}}

\newcommand\R{\mathbb R}

\newcommand\br{\begin{remark}}
\newcommand\er{\end{remark}}
\newcommand\bp{\begin{pmatrix}}
\newcommand\ep{\end{pmatrix}}
\newcommand{\be}{\begin{equation}}
\newcommand{\ee}{\end{equation}}
\newcommand\ba{\begin{equation}\begin{aligned}}
\newcommand\ea{\end{aligned}\end{equation}}

\newcommand{\bap}{\begin{app}}
\newcommand{\eap}{\end{app}}
\newcommand{\begs}{\begin{exams}}
\newcommand{\eegs}{\end{exams}}
\newcommand{\beg}{\begin{example}}
\newcommand{\eeg}{\end{exaplem}}
\newcommand{\bpr}{\begin{proposition}}
\newcommand{\epr}{\end{proposition}}
\newcommand{\bt}{\begin{theorem}}
\newcommand{\et}{\end{theorem}}
\newcommand{\bc}{\begin{corollary}}
\newcommand{\ec}{\end{corollary}}
\newcommand{\bl}{\begin{lemma}}
\newcommand{\el}{\end{lemma}}
\newcommand{\bd}{\begin{definition}}
\newcommand{\ed}{\end{definition}}
\newcommand{\brs}{\begin{remarks}}
\newcommand{\ers}{\end{remarks}}

\newcommand{\diag}{{\rm diag }}

\newtheorem{theorem}{Theorem}[section]
\newtheorem{proposition}[theorem]{Proposition}
\newtheorem{corollary}[theorem]{Corollary}
\newtheorem{lemma}[theorem]{Lemma}

\theoremstyle{remark}
\newtheorem{remark}{Remark}
\theoremstyle{definition}
\newtheorem{definition}[theorem]{Definition}

\newtheorem{example}[theorem]{Example}

\newcommand{\ti}{\tilde}

\newcommand{\spm}{{\ensuremath{{\scriptscriptstyle\pm}}}}
\renewcommand{\sp}{{\ensuremath{{\scriptscriptstyle +}}}}
\newcommand{\sm}{{\ensuremath{{\scriptscriptstyle -}}}}

\title{Balanced flux formulations for multidimensional Evans function computations for viscous shocks}
\author{Blake Barker}
\address{Department of Mathematics, Brigham Young University, Provo, UT 84603}
\email{bhbarker@indiana.edu}
\thanks{B.B. was partially supported by NSF grant DMS-0801745}
\author{Jeffrey Humpherys}
\address{Department of Mathematics, Brigham Young University, Provo, UT 84603}
\email{jeffh@math.byu.edu}
\thanks{J.H. was partially supported by NSF grant DMS-0847074 (CAREER)}
\author{Gregory Lyng}
\address{Department of Mathematics, University of Wyoming, Laramie, WY 82071}
\email{glyng@uwyo.edu}
\thanks{G.L. was partially supported by NSF grants DMS-0845127 (CAREER) and DMS-1413273}
\author{Kevin Zumbrun}
\address{Department of Mathematics, Indiana University, Bloomington, IN 47405}
\email{kzumbrun@indiana.edu}
\thanks{K.Z. was partially supported by NSF grant DMS-0801745}
\date{Last Updated:  \today}
\begin{document}
\begin{abstract}
The Evans function is a powerful tool for the stability analysis of viscous shock profiles; zeros of this function carry stability information. In the one-dimensional case, it is typical to compute the Evans function using Goodman's integrated coordinates \cite{G1}; this device facilitates the search for zeros of the Evans function by winding number arguments.    
Although integrated coordinates are not available in the multidimensional case, we show here that there is a choice of coordinates which gives similar advantages.
\end{abstract}
\maketitle
\tableofcontents

\section{Introduction}\label{sec:intro}
\subsection{Overview}\label{ssec:over}
The Evans function has proven to be a potent theoretical and numerical tool for the stability analysis of viscous shock profiles; see, e.g.,   \cite{HLyZ,Z1}.  In the multidimensional setting, the Evans function $D$ is a function of frequencies $(\lam,\xi)\in\{\Re\lam\geq0\}\times\RR^{d-1}$ where the complex spectral parameter $\lam$ is dual to time and the vector $\xi$ is dual to the transverse spatial directions. Zeros of $D$ with $\Re\lam>0$ correspond to perturbations that grow exponentially in time. A central task in the stability analysis of viscous shock profiles is therefore the determination of the number and location of zeros (if any) of $D$ in the unstable half space. Indeed, generalized spectral stability---roughly, the absence of such zeros---is a sufficient condition for nonlinear asymptotic stability with explicit algebraic-in-time rates of decay in $L^p$, $p\geq 2$; see \cite{Z1}. The precise statement of generalized spectral stability is formulated in terms of the Evans function itself. 
For important physical problems, e.g., gas dynamics or magnetohydrodynamics, locating zeros of the Evans function and verifying this Evans-function condition is a task that requires the numerical approximation of $D$. 

We discuss here one important practical aspect of computing Evans functions associated with viscous shock profiles for multidimensional systems of conservation laws with physically appropriate ``real'' or partially parabolic viscosity. In particular, we develop the basic properties of various formulations of the Evans function based on particualr choices of the phase variables in the first-order formulation of the associated eigenvalue problem. 
We call these the flux, the balanced flux, and the modified balanced flux formulations, and we show that these formulations have concrete benefits for the numerical computation of $D$. The flux coordinates have their origins in the work of Goodman \cite{G1,G2}, and the balanced flux formulation was originally introduced by Plaza \& Zumbrun \cite{PZ} for the purpose of analyzing the spectral stability of small-amplitude multidimensional relaxation shocks. Here, we propose a further modification of these coordinates that preserves the desirable property of analyticity with respect the complex eigenvalue parameter. 

To put these developments in context, we recall that in one space dimension it is a standard practice to use Goodman's tactic of integrated coordinates \cite{G1}. Importantly, this maneuver removes the translational eigenvalue at the origin and is advantageous both for energy estimates \cite{HZ} and for computation of the Evans function \cite{BHRZ,HLyZ}. 
We show that our balanced flux coordinates and modified balanced flux coordinates also have this desirable property. Indeed, we give  a new, transparent proof of this fact which recovers Zumbrun \& Serre's fundamental link \cite{ZS} between low-frequency behavior of the (viscous) Evans function and the (inviscid) Lopatinski\u\i\ determinant (see Proposition \ref{lem:zs} below). 
Moreover, unlike the balanced flux formulation, our modified version retains analyticity with respect to $\lambda$ while still reducing to the usual integrated Evans function when $\xi=0$ (the one-dimensional case). For each constant $\xi$-slice of frequency space, this is perhaps the truest generalization of the integrated Evans function to the multidimensional setting. However, some of the radial uniformity is lost. 

Both of our balanced flux formulations accommodate multidimensional systems with real viscosity and are therefore applicable to physical systems such as gas dynamics or magnetohydrodynamics. Indeed, one significant benefit of the flux framework presented here is that it provides a systematic choice of ``good'' coordinates for Evans-function computations for the stability of viscous shock profiles. We note that, for example, the coordinates used for the Evans-function computations for one-dimensional gas dynamics in \cite{HLyZ} were created on an ad hoc basis. That is, though they were based on Goodman's integrated coordinates, the actual construction relied heavily on the precise form of the equations of compressible gas dynamics. Here, under minimal hypotheses, we show---for a broad class of equations in one and several space dimensions---that there is a choice of coordinates that accommodates real viscosity, maintains analyticity with respect to the spectral paramter $\lambda$, and removes the translational eigenvalue at the origin.  These features are all important for the practical treatment of shock stability by numerical computation of the Evans function. 
Finally, the utility of our proposed formulation is demonstrated by numerical computations of the Evans function for multi-dimensions 
in the fundamental settings of isentropic \cite{BHLZ} and nonisentropic \cite{HLyZ2} gas dynamics and MHD \cite{BMZ}; collectively, these are the first successful multi-dimensional Evans-function computations
for viscous shock waves.

\subsection{Plan}\label{ssec:plan}
In \S\ref{sec:prelim} we establish the setting of our analysis. Namely, we outline the rather general framework of hyperbolic--parabolic systems of conservation laws to which our flux formulations apply. In \S\ref{sec:evans} we describe the balanced flux formulation and its modification. For the convenience of the reader, we describe the full development in both the important special case $d=1$ and in the general case $d>1$. This slight repetition allows us to highlight the connection between the balanced flux form and the oft-used integrated  coordinates in one dimension. We establish the main result of the paper---a proof detailing the low-frequency behavior of the flux forms---in \S\ref{sec:lowfreq}. 
Finally, in conclusion, we discuss some practical consequences in \S\ref{sec:practical}.
Appendix \ref{sec:b21n0} outlines the generalization of integrated coordinates to the case that the lower-left-hand block of the viscosity matrices does not vanish; see Remark \ref{rem:b21}. 

\section{Preliminaries}\label{sec:prelim}
\subsection{Conservation laws with viscosity}\label{ssec:claws}
A number of physical systems take the form of conservation laws with partially parabolic or ``real'' viscosity. That is, they are partial differential equations of 
 block hyperbolic--parabolic type with form
$$
f^0(U)_t+\sum_{j=1}^d f^j(U)_{x_j}=\sum_{j,k=1}^d (B^{jk}(U)U_{x_k})_{x_j}\,.
$$
Here, $x=(x_1,\ldots,x_d)\in\RR^d$, $t\in\RR$, and $U\in\RR^n$ with
\[
f^j:\RR^n\to\RR^n\,,\;j=0,1,\ldots,d\,.
\]
The $d^2$ viscosity matrices $B^{jk}\in\RR^{n\times n}$ are each assumed to have the block structure
\beq\label{eq:bblock}
B^{jk}(U)=
\bpm
0 & 0 \\ 0 & b^{jk}(U)
\epm\,.
\eeq
The blocks in \eqref{eq:bblock} have  sizes
\beq\label{eq:bs}
\left(
\begin{array}{c|c}
r\times r & r\times (n-r) \\  \hline
(n-r)\times r & (n-r)\times(n-r)
\end{array}\right)\,,
\eeq
and we write $U$ as 
\beq\label{eq:ublock}
U=\bpm u_1 \\ u_2\epm\,,\quad u_1\in\RR^r\,,u_2\in\RR^{(n-r)}\,,
\eeq
to respect this block structure. 
We write $A^j(U):=\dif f^j(U)$ for $j=0,1,\ldots,d$, and, when necessary, we write any $n\times n$ matrix $M$ in block form
\beq
M=\bpm
M_{11} & M_{12} \\
M_{21} & M_{22}
\epm\,,
\eeq
with block sizes as in \eqref{eq:bs}. We also write $f^j_\ell$ with $\ell=1$ (or $2$) to denote the first $r$ (or the last $n-r$) component functions of the flux $f^j$.

Our interest is in the stability of planar viscous shock profiles. 
Thus, we consider traveling-wave solutions of the form
\beq
U(x,t)=\bar U(x_1-st)\,,\; \lim_{z\to\pm\infty} \bar U(z)=U_\pm \,,
\eeq
and, given our interest in the stability of these waves, 
our first step is to transform to moving coordinates
$\tilde x_1=x_1-st$ in which $\bar U$ becomes stationary. 
This gives (dropping tildes) the modified system of equations
\beq\label{eq:claw}
f^0(U)_t + \big(f^1(U)-sf^0(U)\big)_{x_1}
+\sum_{j=2}^d f^j(U)_{x_j}=\sum_{j,k=1}^d (B^{jk}(U)U_{x_k})_{x_j}\,.
\eeq
We make the structural assumptions
\beq\label{eq:a11inv}
\det (A^1_{11}(\bar U) -s A^0_{11}(\bar U)) \neq 0\, 
\; \text{(hyperbolic noncharacteristicity)}
\tag{H1}
\eeq
and
\beq\label{eq:bparab}
\sigma\left(\sum \eta_j\eta_k b^{jk}(\bar U)\right)
\ge \theta |\eta|^2, \,
\theta>0,\; \text{for all} \; \eta=(\eta_1,\ldots,\eta_d)\in \RR^d
\; \text{(parabolicity).}
\tag{H2}
\eeq

\begin{remark}\label{rem:b21}
Our structural conditions apply in complete generality to the principal equations
of continuum mechanics: compressible gas dynamics, MHD, and viscoelasticity.
The methods described here can be extended to the case that the viscosity matrices have nonzero lower left-hand blocks, i.e.,
\[B^{jk}(U)=
\bpm
0 & 0 \\
b_{21}^{jk}(U) & b_{22}^{jk}(U) 
\epm\,
,
\]
under \eqref{eq:bparab} and the modified hyperbolic condition
\beq\label{eq:a11invalt}
\det (A^1_{11}-A^1_{12}(b_{22}^{11})^{-1}b_{21}^{11} -s A^0_{11}) (\bar U) \neq 0\, ,
\tag{$\text{H1}^\prime$}
\eeq
introducing an ``approximate parabolic 
coordinate'' $\check u_2= u_2 + b_{22}^{-1}b_{21}u_2$ 
similarly as in \cite{MaZ1,MaZ2}.
This is essential, for example, 
if there does not exist a true parabolic variable,
i.e., $b_{12}\partial_{u_1} + b_{22}\partial_{u_2}$ is not a matrix multiple
of $\nabla_U \tilde u_2(U)$, $\tilde u_2\in \RR^{n-r}$ for some 
``exact parabolic coordinate'' $\tilde u_2$.
It can be useful also if it is more convenient to work with a coordinate other
than the true parabolic one.
However, in practice we find it more convenient to work with the actual 
parabolic variable, and so, to simplify the presentation, we will 
restrict our attention to the main case \eqref{eq:bblock}, \eqref{eq:a11inv},
\eqref{eq:bparab}.
We briefly treat the more general case (in one space dimension) in Appendix \ref{sec:b21n0}.
\end{remark}

\subsection{Standing-wave profiles \& the eigenvalue problem}\label{ssec:swp}
\subsubsection{Profile solutions}\label{sssec:profiles}
Examining \eqref{eq:claw}, we see that the standing wave 
$\bar U$ must satisfy the ordinary differential equation ($'=\dif/\dif x_1$)
\beq\label{eq:standing}
\tilde f^1(\bar U)'=
(B^{11}(\bar U)\bar U')'\,,
\qquad
\tilde f^1(\bar U):= f^1(\bar U)-sf^0(\bar U)\,.
\eeq
Evidently, equation \eqref{eq:standing} can be integrated once to 
\beq\label{eq:twode}
B^{11}U'=\tilde f^1(\bar U)-\tilde f^1(U_\sm)\,.
\eeq
Note that, using block structure, we may rewrite \eqref{eq:twode} as 
\begin{subequations}\label{eq:tw}
\begin{align}
0 & = \tilde f^1_1(\bar U)-\tilde f^1_1(U_\sm)\,, \label{eq:alg}
 \\
b^{11}(\bar U) \bar u_2' &=\tilde f^1_2(\bar U)-\tilde f^1_2(U_\sm)\,.\label{eq:dif}
\end{align}
\end{subequations}
We expect that the algebraic equation \eqref{eq:alg} defines a submanifold of $\RR^n$ on which \eqref{eq:dif} defines a flow. To solve for $u_1$ in terms of $u_2$, locally at least, the implicit function theorem requires that 
$\det  \tilde A^1_{11}(\bar U) \neq 0$, or, equivalently, 
\eqref{eq:a11inv}.
This motivates the introduction of assumption \eqref{eq:a11inv}.

An obvious necessary condition for the existence of a traveling-wave connection is that the end states $U_\pm$ must be equilibria of \eqref{eq:twode}. Therefore, from \eqref{eq:tw}, we obtain immediately the Rankine--Hugoniot condition
\beq\label{eq:rh}\tag{RH}
\tilde f^1(U_\sp)-\tilde f^1(U_\sm)=0\,.
\eeq

\br[Hyperbolic classification]\label{rem:lax}
We denote by $i_\sp$ the number of characteristics incoming to the shock from the right and by $i_\sm$ the number  of characteristics incoming from the left. We write $i:=i_\sp+i_\sm$ for the total number of incoming characteristics.
Then, the hyperbolic classification of $\bar U(\cdot)$, i.e., the classification of the associated hyperbolic shock $(U_\sm, U_\sp)$, is given in the table below. 
\begin{center}
\begin{tabular}{c|l}
Shock Type & $i$ \\ \hline 
Lax  & $i=n+1$ \\
Undercompressive (u.c.) & $i\leq n$ \\
Overcompressive (o.c.) &  $i \geq n+2$
\end{tabular}
\end{center}
\er
\subsubsection{Linearization, eigenvalue problem}
Supposing, as above, that $U(x,t)=\bar U(x_1)$ 
is a steady solution of \eqref{eq:claw}, we
linearize about $\bar U$ to obtain an equation describing the approximate 
evolution of a perturbation also called $U=U(x,t)$. The linearized equations for $U$ read
\beq\label{eq:meval}
\bar A^0U_t+\sum_{j=1}^d(\bar A^j U)_{x_j}=\sum_{j,k=1}^d(\bar B^{jk} U_{x_k})_{x_j}\,,
\eeq
where
\[
\bar A^1U:=A^1(\bar U)U-sA^0(\bar U)-\dif B^{11}(\bar U)(U,\bar U_{x_1})
\]
and
\[
\bar A^0:=A^0(\bar U)\,,\quad \bar A^jU:=A^j(\bar U)U-\dif B^{j1}(\bar U)(U,\bar U_{x_1})\,,(j \neq 1)\,,\quad \bar B^{jk}:=B^{jk}(\bar U)\,.
\]
Taking the Laplace transform in time (dual variable $\lambda$) and Fourier transform (dual variable $\xi=(\xi_2,\ldots,\xi_d)$) in the transverse spatial directions $(x_2,\dots,x_d)$, finally, we obtain
the generalized eigenvalue equation 
\begin{multline}\label{eq:mgeval}
\lambda \bar A^0U+
(\bar A^1 U)' + \sum_{j=2}^d\mi\xi_j \bar A^j U
= (\bar B^{11} U')' + \sum_{k= 2}^d (\mi\xi_k \bar B^{1k} U)' \\
+\sum_{j=2}^d\mi\xi_j\bar B^{j1}U'
-
\sum_{j, k=2}^d\xi_j\xi_k \bar B^{jk} U\,.
\end{multline}
In \eqref{eq:mgeval} we have now used $U=U(x_1,\lam, \xi)$ to represent the transformed perturbation. Now, a basic criterion for stability of the viscous profile $\bar U$ is that the eigenvalue equation \eqref{eq:mgeval} should have no solutions which decay at $x_1=\pm\infty$ with $(\lambda,\xi)\in \{\Re\lam>0\}\times\RR^{d-1}$. Searching for such values of $(\lambda,\xi)$ is the \emph{spectral stability problem}, and the Evans function $D=D(\lam,\xi)$ vanishes precisely at such values. 
Our focus, then, is on locating zeros (if any) of $D$. 
It is clearly advantageous to design an Evans function with as much structure as possible to aid the search for unstable zeros. 
For example, analyticity is valuable; it allows the search for zeros to
proceed using the argument principle.

\subsubsection{First-order systems}
One may visualize the the construction of the Evans as follows.
The basic set-up is based on reformulating the eigenvalue problem \eqref{eq:mgeval} as a first-order system of differential equations
\beq\label{eq:1order}
W'=\mat{A}(x_1;\lam,\xi)W\,.
\eeq
We note that our block structure hypotheses imply that $\mat{A}$ is an $N\times N$ matrix with $N=(2n-r)$. If $\bar U$ decays rapidly to its limiting values $U_\spm$ as $x_1\to\pm\infty$, then the coefficient matrix $\mat{A}$ should also have constant (with respect to $x_1$) limiting values. We denote these by $\mat{A}^\spm(\lam,\xi)$. 

Then, the Evans function is built out of the subspaces of solutions of \eqref{eq:1order} which grow at $-\infty$ and decay at $+\infty$; the construction of these subspaces starts with an analysis of the constant-coefficient limiting system $W'=\mat{A}^\spm(\lam,\xi)W$.
That is, if the collection
$\{W_1^\sp,\ldots,W_k^\sp\}$ forms a basis for the solutions of \eqref{eq:1order} that decay at $+\infty$ and, similarly, $\{W_{k+1}^\sm,\ldots W_{N}^\sm\}$ spans the solutions that grow at $-\infty$, the Evans function can be written as 
\beq\label{eq:abstractevans}
D(\lam, \xi):=\det(W_1^\sp,\ldots,W_k^\sp,W_{k+1}^\sm,\ldots W_{N}^\sm)|_{x_1=0}\,.
\eeq
Evidently, if $D(\lam_\circ, \xi_\circ)=0$, then \eqref{eq:abstractevans} shows that there is a linear dependence between these two subspaces. But then there must be a solution which decays at both $\pm\infty$, an eigenfunction. 

\br
Clearly, different choices of bases lead to distinct Evans functions, and the Evans function is highly non unique. However, the construction guarantees that each representative chosen from the family of Evans functions has the fundamental property that it vanishes at eigenvalues of \eqref{eq:mgeval}. Indeed, in a companion paper, we discuss how differing coordinate systems at the level of original partial differential equation \eqref{eq:claw} influence the character of the resulting Evans function(s) \cite{BHLZ}. A related issue is the previously mentioned use of integrated coordinates and the ability to manipulate the character of the Evans function through the formulation of the first-order system \eqref{eq:1order}. For example, there are several choices of the phase variable $W$. Even though all of these Evans functions carry the same stability information, different versions may be more amenable to analysis or computation in various regimes/settings. For example, when counting zeros by the argument principle, it may be useful to limit excessive winding and unwinding. 
\er

\section{Formulating the Evans function}\label{sec:evans}

\subsection{Flux variables and integrated coordinates ($d=1$)}\label{ssec:flux1d}
For its independent interest and to showcase the relationship between the flux and balanced flux variables we introduce below and the integrated coordinates that are commonly used in one-dimensional Evans-function calculations, in this subsection we specialize to a single space dimension ($d=1$). 
We recall that, when formulated in terms of integrated coordinates, the Evans function has the useful property that, for Lax or overcompressive shocks, it does not vanish at the origin. 
Among other benefits, this feature is useful for the practical computation of the Evans function. While integrated coordinates do not naturally generalize to the multidimensional setting, the flux and balanced flux forms do; we describe this generalization to the case $d> 1$ below in Section \ref{s:multid}.

\subsubsection{Integrated coordinates}\label{sec:integrate}
In the case $d=1$, the linearized equation \eqref{eq:mgeval} collapses ($\xi=0$), and 
we may write the associated eigenvalue equation as 
\beq\label{eq:eval1d}
\lambda \bar A^0U+(\bar A^1 U)'=(\bar B^{11} U')'\,.
\eeq
To obtain the integrated Evans function, we define 
\be\label{w}
w:=\bar A^0U, \quad W'=w,
\ee
and we find, integrating \eqref{eq:eval1d},
\be\label{Weq}
\lambda W +\bar A^1 (\bar A^0)^{-1} W'
=\bar B^{11}( (\bar A^0)^{-1} W')'\,.
\ee
To write \eqref{Weq} as a first-order system, we set 
$$
Z:=\bp  W\\ (0,I_{n-r}) (\bar A^0)^{-1}W'\ep\,.
$$
We thus have $Z'=\mat{A}_\mathrm{int}(x;\lam)Z$, where (denoting by $\bfa$ the inverse of the matrix $\bar A^1_{11}$, recall \eqref{eq:a11inv})
\beq\label{eq:Aint}
\mat{A}_\mathrm{int}(x;\lam)=
\bp
-\lambda \bar A^0_{11}\bfa &  0 &  \bar A^0_{12}-\bar A^0_{11}\bfa \bar A^1_{12} \\
-\lambda \bar A^0_{21}\bfa & 0  & \bar A^0_{22}-\bar A^0_{21}\bfa \bar A^1_{12} \\
-\lambda (\bar b^{11})^{-1}\bar A^1_{21}\bfa & \lambda(\bar b^{11})^{-1} &
(\bar b^{11})^{-1}(\bar A^1_{22}-\bar A^1_{21}\bfa\bar A^1_{12})
\ep
\eeq
 is obtained by
solving for $(I,0)(\bar A^0)^{-1}W'$, 
whence, together with the coordinate 
$(0,I)(\bar A^0)^{-1}W'$, we obtain 
$(\bar A^0)^{-1}W'$ and thus $W'$.
For this step, multiply \eqref{Weq} by $(I,0)$ to obtain
$$
\lambda W_1 = - (\bar A^1_{11},\bar A^1_{12})(\bar A^0)^{-1}W'\,,
$$
from which we see that, provided $\bar A^1_{11}$ is invertible
(same assumption needed for flux variables, and indeed even for framing via implicit function theorem of
the profile equation; see discussion below \eqref{eq:tw}), we can solve for 
$(I,0)(\bar A^0)^{-1}W'$ in terms of the known coordinates
$W$ and $(0,I)(\bar A^0)^{-1}W'$ of $Z$.
We shall not carry out this computation in detail, as we shall
reproduce it by an equivalent and somewhat simpler derivation below.

\subsubsection{Flux variables}

We now describe an alternative way to write the eigenvalue equation as a first-order system. 
To write the eigenvalue equation in flux variables, we observe that \eqref{eq:eval1d} can be rewritten as 
\beq
\lambda \bar A^0 U=(\bar B^{11} U'-\bar A^1U)'\,,
\eeq
which motivates the definition of the \emph{flux variable}
\beq\label{eq:f}
\f:=\bar B^{11} U'-\bar A^1U\,,
\eeq
or
\begin{subequations}
\begin{align}
\f_1&=-\bar A^1_{11}u_1-\bar A^1_{12}u_2\,, \label{eq:f1}\\
\f_2&=\bar b^{11}u_2'-\bar A^1_{21}u_1-\bar A^1_{22}u_2 \label{eq:f2}\,,
\end{align}
\end{subequations}
with
\beq
\f'=\lambda \bar A^0 U.
\eeq
Provided that the $r\times r$ matrix $\bar A^1_{11}$ is invertible---as 
assumed in \eqref{eq:a11inv},
we see immediately from \eqref{eq:f1} that
\be\label{eq:solve}
u_1=-(\bar A_{11}^1)^{-1}\big(\f_1+\bar A^1_{12}u_2\big)\,.
\ee
Thus, using \eqref{eq:solve}, we may write \eqref{eq:eval} as a first-order system in flux variables as 
\be\label{eq:fode}
W'=\mat{A}_\f(x_1;\lambda)W\,,
\eeq
with
\beq
W=\bpm
\f \\ u_2
\epm\,,
\eeq
and the coefficient matrix $\mat{A}_\f$  given by 
%
\be\label{eq:fluxA}
\mathbf{A}_\f(x_1;\lam)=
\bp
-\lambda \bar A^0_{11}\bfa & 0 & \lambda(\bar A^0_{12}-\bar A^0_{11}\bfa \bar A^1_{12}) \\
-\lambda \bar A^0_{21}\bfa & 0 & \lambda(\bar A^0_{22}-\bar A^0_{21}\bfa\bar A^1_{12}) \\
-(\bar b^{11})^{-1}\bar A^1_{21}\bfa &
(\bar b^{11})^{-1}&
(\bar b^{11})^{-1}( \bar A^1_{22} - \bar A^1_{21}\bfa \bar A^1_{12})
\ep.
\ee

\begin{remark}
Perhaps the quickest route to the form of $\mat{A}_\f$ comes from multiplying \eqref{eq:f} from the left by the matrix 
\beq\label{eq:schur}
\bpm
(\bar A^1_{11})^{-1} & 0 \\
-\bar A^1_{21}(\bar A_{11}^1)^{-1} & I_{n-r}
\epm\,,
\eeq
from which we immediately obtain 
\beq\label{eq:simple}
\bpm
0 & 0 \\
0 & \bar b^{11}
\epm 
\bpm
u_1' \\ u_2'
\epm
-\bpm
I & (\bar A^1_{11})^{-1}\bar A_{12}^1 \\
0 & \bar A^1_{22}-\bar A_{21}^1(\bar A_{11}^1)^{-1}\bar A_{12}^1 
\epm
\bpm u_1 \\ u_2 \epm
=
\bpm
(\bar A^1_{11})^{-1}\f_1 \\
-\bar A_{21}^1(\bar A_{11}^1)^{-1}\f_1+\f_2
\epm\,.
\eeq
Observe that the first row of \eqref{eq:simple} gives \eqref{eq:solve}, and the row operation in \eqref{eq:schur} has eliminated $u_1$ from the second row. 
\end{remark}

\subsubsection{Balanced flux variables}\label{sssec:bf}
Introducing the balanced flux variable
\beq\label{eq:bflux}
\f^\sharp:= \lambda^{-1}(\bar B^{11} U'-\bar A^1U)=\frac{\f}{\lambda}\,,
\eeq
and
\beq
 W^\sharp=\bpm
\f^\sharp \\ u_2
\epm\,,
\eeq
effects a scaling transformation on \eqref{eq:fode}; we find in this case that the eigenvalue ODE can be written as
\be\label{eq:bfode}
(W^\sharp)'=\mat{A}_{b\f}(x_1;\lambda) W^\sharp\,,
\eeq
with
\beq\label{eq:bf}
\mat{A}_{b\f}(x_1;\lam)=
\bp
-\lambda \bar A^0_{11}\bfa & 0 & \bar A^0_{12}-\bar A^0_{11}\bfa \bar A^1_{12} \\
-\lambda \bar A^0_{21}\bfa & 0 & \bar A^0_{22}-\bar A^0_{21}\bfa\bar A^1_{12} \\
-\lambda (\bar b^{11})^{-1}\bar A^1_{21}\bfa &
\lambda(\bar b^{11})^{-1} &
(\bar b^{11})^{-1}(\bar A^1_{22}-\bar A^1_{21}\bfa \bar A^1_{12})
\ep
\eeq
identical to $ \mat{A}_\mathrm{int}$.
That is, \emph{the balanced flux and integrated formulations exactly agree.}

\medskip

\br
Noting that the balanced flux variable $\f^\sharp$ satisfies
\beq\label{eq:equivphase}
(\f^\sharp)'= \lambda^{-1}(\bar B U'-\bar A^1U)'
=\bar A^0 U=w,
\eeq
we see that $(\f^\sharp)'$ and $W'$ agree, so that the ODEs must be equivalent.
Moreover, keeping in mind the relation \eqref{eq:equivphase},
one may check directly that the two described derivations coincide.

\er
\subsection{Flux, balanced flux, and modified balanced flux variables $(d>1)$} \label{s:multid}
\subsubsection{Flux Variables $(d>1)$}
We now proceed to describe the flux formulation for the multidimensional system \eqref{eq:claw}. The starting point is the eigenvalue equation \eqref{eq:mgeval}.
First, we note that 
\[
\left(\sum_{j=2}^d\mi\xi_j\bar B^{j1}U\right)'=\sum_{j=2}^d\mi\xi_j(\bar B^{j1})'U+\sum_{j=2}^d\mi\xi_j\bar B^{j1}U'\,,
\]
so that we may rearrange \eqref{eq:mgeval} to 
\begin{equation}\label{eq:mgeval2}
\lambda \bar A^0U+ \sum_{j=2}^d\mi\xi_j \tilde A^j U
+
\sum_{j, k=2}^d\xi_j\xi_k \bar B^{jk} U
\\
= \left(\bar B^{11} U' + \sum_{j= 2}^d \mi\xi_j \bar B^{j} U-\bar A^1U\right)'\,,
\end{equation}
where 
$\tilde A^j:=\bar A^j+(\bar B^{j1})'$
and
$\bar B^j:=\bar B^{j1}+\bar B^{1j}$.
Thus, we may define the \emph{flux variable} $\f$ by 
\beq\label{eq:mf}
\f:=\bar B^{11} U'+ \sum_{j=2}^d \mi\xi_j \bar B^{j} U-\bar A^1U\,, 
\eeq
and the goal is to recast \eqref{eq:mgeval2} as a first-order system.
$
W'=\mathbf{A}_\f(x_1;\lambda,\xi)W
$
with 
\[
W=\bpm
\f \\
u_2
\epm\,.
\]
We write $\bar B^{\xi}:=\sum_{j\neq 1}\xi_j\bar B^j$ (and similarly, $\tilde A^{\xi}:=\sum_{j\neq1}\xi_j\tilde A^j$, $\bar B^{\xi\xi}$, and so on), and we note that our block structure assumption implies that $\bar B^{\xi}$ has the form
\[
\bar{B}^{\xi}=
\bpm
0 & 0 \\ 
0 & \bar b^{\xi}
\epm\,.
\]
Thus, we may perform a simplifying row operation on \eqref{eq:mf}; we multiply on the left by 
\beq\label{eq:schur2}
\bpm
(\bar A^1_{11})^{-1} & 0 \\
-\bar A^1_{21}(\bar A_{11}^1)^{-1} & I
\epm\,,
\eeq
Equation \eqref{eq:mf} then becomes
\begin{multline}\label{eq:simple2}
\bpm
0 & 0 \\
0 & \bar b^{11}
\epm 
\bpm
u_1' \\ u_2'
\epm
+
\bpm
0 & 0 \\
0 &\mi\bar{b}^{\xi}
\epm
\bpm 
u_1 \\
u_2
\epm
-\bpm
I & (\bar A^1_{11})^{-1}\bar A_{12}^1 \\
0 & \bar A^1_{22}-\bar A_{21}^1(\bar A_{11}^1)^{-1}\bar A_{12}^1 
\epm
\bpm u_1 \\ u_2 \epm
\\
=
\bpm
(\bar A^1_{11})^{-1}\f_1 \\
-\bar A_{21}^1(\bar A_{11}^1)^{-1}\f_1+\f_2
\epm\,.
\end{multline}
Evidently, the third row of $\mathbf{A}_\f$ can be read off from \eqref{eq:simple2}. In addition, \eqref{eq:simple2} contains the fundamental identity
\beq\label{eq:solve2}
-u_1-(\bar A^1_{11})^{-1}\bar A^1_{12}u_2=(\bar A^1_{11})^{-1}\f_1
\eeq
which allows us to eliminate $u_1$ in favor of $u_2$ and $\f_1$.
To obtain the first two rows of $\mathbf{A}_\f$, we write
\beq
\f'=
\lambda \bar A^0U+ \mi\tilde A^{\xi} U
+ \bar B^{\xi\xi} U
\eeq
in terms of components, so that 
\begin{subequations}\label{eq:mdf}
\begin{align}
\f_1'&=\lambda\bar A_{11}^0u_1+\lambda \bar A^0_{12}u_2+\mi\tilde A^{\xi}_{11}u_1+\mi\tilde A^{\xi}_{12}u_2\,,  \\
\f_2'&=\lambda\bar A_{21}^0u_1+\lambda \bar A^0_{22}u_2+\mi\tilde A^{\xi}_{21}u_1+\mi\tilde A^{\xi}_{22}u_2+\bar b^{\xi\xi}u_2\,.
\end{align}
\end{subequations}
Thus, we may use \eqref{eq:solve2} to eliminate $u_1$ from \eqref{eq:mdf} and obtain $\mathbf{A}_\f$. Thus, continuing to denote $(A_{11}^1)^{-1}$ by $\mathbf{a}$, we find that
\beq
\mathbf{A}_\f 
=
\bpm
-\lambda\bar A^0_{11}\mathbf{a}-\mi\tilde A_{11}^{\xi}\mathbf{a} 
& 
0 
& -\lambda(\bar A^0_{12}\mathbf{a}\bar A^1_{12}-\bar A^0_{12})-\mi\tilde A_{11}^{\xi}\mathbf{a}\bar A^1_{12}+\mi\tilde A_{12}^{\xi}
\\
-\lambda\bar A^0_{21}\mathbf{a}-\mi\bar A_{21}^{\xi}\mathbf{a} 
& 
0 
& -\lambda(\bar A^0_{21}\mathbf{a}\bar A^1_{12}-\bar A^0_{22})-\mi\tilde A_{21}^{\xi}\mathbf{a}\bar A^1_{12}+\mi\tilde A_{22}^{\xi}+\bar b^{\xi\xi}
\\
-(\bar b^{11})^{-1}(\bar A^1_{21}\mathbf{a})
&
(\bar b^{11})^{-1}
& 
(\bar b^{11})^{-1}(\bar A^1_{22}-\bar A^1_{21}\mathbf{a}\bar A^1_{12}-\bar b^{\xi})
\epm.
\eeq

\begin{remark}
This may readily be seen to be equivalent to the usual Evans function
of \cite[Eq. (3.1), p. 356]{Z1}, based on variable $(u, b_{11}u_1'+b_{22}u_2')$, since the phase variables for the two Evans functions
are conjugate by a frequency-independent coordinate transformation.
\end{remark}

\subsubsection{Balanced flux variables $(d>1)$}\label{ssec:balance}
Proceeding from the flux form, we may, similarly as in \S\ref{sssec:bf} above, define a multi-dimensional
Evans function analogous to the integrated Evans function used in
one dimension.
We define 
\beq
\f^\sharp:=\f/r(\lambda,\xi),\ \lambda^{\sharp} = \lambda/r(\lambda,\xi),\ \mathrm{and}\ \xi^{\sharp} = \xi/r(\lambda,\xi)
\eeq
with
\beq\label{rbf}
r(\lambda,\xi):=|\lambda,\xi|,
\eeq
to obtain
an alternative Evans function.
In this case, the first-order system takes the form 
$
\tilde W'=\mathbf{A}^\sharp(x_1;\lambda,\xi)\tilde W
$
with $\ti W=(\f^\sharp,u_2)^\tr$ and 
\beq\label{eq:multidBf}
\mathbf{A}^\sharp
=
\bpm
r(-\lambda^\sharp\bar A^0_{11}\mathbf{a}-\mi\tilde A_{11}^{\xi}\mathbf{a}) 
& 
0 
& -\lambda^\sharp(\bar A^0_{12}\mathbf{a}\bar A^1_{12}-\bar A^0_{12})-\mi\tilde A_{11}^{\xi^\sharp}\mathbf{a}\bar A^1_{12}+\mi\tilde A_{12}^{\xi^\sharp}
\\
r(-\lambda^\sharp\bar A^0_{21}\mathbf{a}-\mi\bar A_{21}^{\xi^\sharp}\mathbf{a}) 
& 
0 
& -\lambda^\sharp(\bar A^0_{21}\mathbf{a}\bar A^1_{12}-\bar A^0_{22})-\mi\tilde A_{21}^{\xi^\sharp}\mathbf{a}\bar A^1_{12}+\mi\tilde A_{22}^{\xi^\sharp}+r \bar b^{\xi^\sharp\xi^\sharp}
\\
-r(\bar b^{11})^{-1}(\bar A^1_{21}\mathbf{a})
&
r(\bar b^{11})^{-1}
& 
(\bar b^{11})^{-1}(\bar A^1_{22}-\bar A^1_{21}\mathbf{a}\bar A^1_{12}-r\bar b^{\xi^\sharp})
\epm
.
\eeq
This determines an Evans function 
$D^\sharp(r,\xi^\sharp, \lambda^\sharp)$ in the
variables $(r,\xi^\sharp,\lambda^\sharp)$, from which we may then extract
an Evans function 
\beq\label{Dbf}
D_{b\f}(\lambda,\xi):=
D^\sharp(r(\lambda,\xi), \xi/r(\lambda,\xi), \lambda/r(\lambda,\xi)).
\eeq

\subsubsection{Modified balanced flux variables $(d>1)$}\label{ssec:mbalance}

Alternatively, we may replace $r$ in \eqref{rbf} with
\beq\label{rmbf}
r_2(\lambda,\xi):=|\xi| + \lambda
\eeq
in the above derivation,
to obtain
an Evans function 
\beq\label{Dmbf}
D_{mb\f}(\lambda,\xi):=
D^\sharp(r_2(\lambda,\xi), \xi/r_2(\lambda,\xi), \lambda/r_2(\lambda,\xi)).
\eeq
that is analytic in $\lambda$, reducing to
the usual (1D) integrated Evans function for $\xi=0$,
and still has the desirable property that it is nonvanishing
at the origin (where it is now multi-valued, depending on limiting
angle).
This is perhaps the truest generalization of the integrated 
Evans function to multi-dimensions, considered $\xi$-slice
by $\xi$-slice.
However, it loses some uniformity in replacing $|\lambda,\xi|$
by the norm-equivalent (for $\Re \lambda \ge 0$) 
quantity $|\xi|+\lambda$.

\section{Low-frequency behavior of the balanced flux forms}\label{sec:lowfreq}
The Evans function $D_{b\f}$
 in the phase variables
$
\tilde W=\bp \f^\sharp\\ u_2\ep
$ is approximately homogeneous near $(0,0)$
and analytic along rays through
the origin (equivalently, when written in polar coordinates).
As noted above, in one dimension, $(\f^\sharp)'=U$, and so
the Evans function $\tilde D(\lam,\xi)$ determined by
phase variables 
$
\bp \f^\sharp\\u_2\ep= \bp \f^\sharp\\(\f^\sharp_2)'\ep
$
is exactly the usual integrated Evans function, with the desirable
property that, for Lax or overcompressive shocks, it does not vanish
at the origin.
We now show that this desirable property persists also for
multi-dimensions.
Recall \cite{M,ZS} that inviscid stability of multidimensional shock waves is determined by
a {\it Lopatinski determinant} $\Delta( \lambda, \xi) $ analogous to the Evans function,
defined on $\xi\in \R$, $\Re \lambda\geq 0$; a shock is {\it uniformly inviscid stable} 
if $\Delta \neq 0$ on $\{\Re\lambda \geq 0\}\setminus \{(0,0)\}$.

\begin{proposition}
\label{lem:zs}
With appropriately chosen bases at $x_1=\pm \infty$,
\be\label{asymptotics}
D_{b\f}(\lambda, \xi)= \gamma \Delta(\check \lambda, \check \xi) 
+o(r),
\ee
 for $r:=|\lambda, \xi|$ sufficiently small,
where 
$(\check \lambda, \check \xi):=r^{-1}(\lambda, \xi)$,
 $\Delta(\lambda, \xi)$ is the inviscid Lopatinski determinant,
and $\gamma$ is a transversality coefficient for the traveling-wave ODE
that is a constant independent of angle $(\check \lambda, \check \xi)$.
In particular, for a uniformly inviscid stable shock, $D_{b\f}(\lambda,\xi)$ has a nonvanishing
limit $\Delta(\check \lambda, \check \xi)$ as $(\lambda,\xi)\to 0$ with $(\check \lambda, \check \xi)$ held fixed.
\end{proposition}

\begin{proof}
This follows from the fact that the component bases of decaying solutions for the standard (``unintegrated'') flux system, when multiplied by 
$\diag \{I, |\lambda,\xi|\}$,
yield a basis for the balanced flux system, {\it except} for ``fast'' decay
elements (corresponding to incoming inviscid modes) which vanish in the $\f$ coordinate and must
be treated instead by multiplying by $\diag \{|\lambda,\xi|^{-1}, I\}$
using L'H\^opital's rule.
But, this means the determinant will change by factor of
$$
|\lambda,\xi|^{o-i}= |\lambda,\xi|^{-c} ,
$$
 where $i$ is the number of incoming and $o$ the
number of outgoing hyperbolic modes, and $c=i-o$ is the ``degree of compressivity'' \cite{ZH,ZS},
equal for Lax or overcompressive shocks to the number of zeros at the origin.

Here, we are using the important fact that fast modes may be chosen 
analytically in $r$ and (locally) in $(\check \lambda,\check \xi)$,
and in the usual flux variables are independent of angle
$(\check \lambda,\check\xi)$, being of form 
$$
\bp 0+ w_1(\check \lambda, \check \xi) r + O(r^2)\\ v\ep.
$$
Thus when multiplied by $\diag \{|\lambda,\xi|^{-1}, I\}$,
they transform to form
$$
\bp  w_1(\check \lambda, \check \xi)  + O(r)\\ v\ep,
$$
which are legitimate basis elements that are still analytic in polar
coordinates $(r, \check\xi,\check \lambda)$.
This validates the choice of fast bases.
Likewise, we may check that the first coordinates $\f$ of slow
modes are chosen as bases for the associated Lopatinski determinant,
so that $\bp \f\\0\ep$ are still independent, and clearly independent
of fast modes, so we still have a basis.
For details on construction of ``fast'' and ``slow'' basis elements in the vicinity of the origin, see \cite{ZH,ZS,Z1}.

We thus have 
$\tilde D(\lambda, \xi)\equiv r^{-c}D(\lambda, \xi)$, 
where $D$ is the usual Evans function
as defined, e.g., in \cite{Z1}, whereupon \eqref{asymptotics} follows
from the fundamental property of the standard Evans function 
\cite{ZS,Z1}, valid for Lax and overcompressive shocks, that
$$
D(\lambda, \xi)=\Delta(\lambda, \xi)+o(r^c)=
r^c \Delta(\check\lambda,\check\xi)+o(r^c).
$$

(For undercompressive shocks
$D(\lambda, \xi)=\Delta(\lambda, \xi)+o(r^c)= r \Delta(\check\lambda,\check\xi)+o(r)$, so that the number of zeros at the origin in balanced flux coordinates
equals or exceeds the number of zeros in standard flux coordinates.)
\end{proof}

\br
The modified balanced flux formulation also removes zeros at the origin, by essentially the same argument, substituting for $|\lambda,\xi|$
the commensurate (for $\Re \lambda \geq 0$) quantity $|\xi|+\lambda$.
\er

\section{Practical Considerations}\label{sec:practical}
\subsection{Kato bases}\label{katorbases}
Proposition \ref{lem:zs} concerns the Evans function induced by a particular choice of (local) bases near $|\lambda,\xi|=0$, or, 
equivalently \cite{GZ,BrZ,HuZ,Z1}, by the choice of initializing stable/unstable eigenbases $R_\pm$ of the limiting coefficient matrices at 
$x\to \pm \infty$, where $R_\pm$ are matrices with columns corresponding to basis elements.
In standard practice, this is done using the {\it Kato ODE}
\be\label{property}
\dot R=PP'R,
\ee
where $P$ is the corresponding stable/unstable eigenprojector \cite{BrZ,HuZ,STABLAB}, and $\dot{ }$ denotes variation along a given path in frequency space $(\lambda,\xi)$.  It can be shown that this is the unique choice such that $P\dot R=0$.
When there exist locally (jointly) analytic bases and projectors $V_\pm$, $P_\pm$ with respect to $(r,\hat \xi, \hat \lambda)$, we can write 
$R_\pm =V_\pm \alpha_\pm$ and use \eqref{property} to derive a linear analytic-coefficient ODE for coefficients $\alpha_\pm$, which are therefore
locally (jointly) analytic in $(r, \hat \xi,\hat \lambda)$ as well.
The change from bases $R_\pm$ to $V_\pm$ changes the value of the resulting Evans function by a nonvanishing analytic factor, hence the conclusions
of Proposition \ref{lem:zs} remain valid for the standard Kato basis as well.
Points where analyticity in $(\hat \xi, \hat \lambda)$ fails at $r=0$ correspond (see \cite{ZS,Z1}) to {\it glancing modes} for the associated inviscid problem, in which the coefficient matrix possesses a Jordan block.  It is shown in \cite{Z1,MeZ} that, at such a point, variations in $r$ and $\lambda$ enter ``together,'' to lowest order as a linear combination in the lower lefthand corner of the standard Jordan form.  For example, a model for a glancing mode/Jordan block of order $2$ is
$$
A(r,\hat \xi,\hat \lambda)= \begin{pmatrix} 0 & 1 \\ \hat \lambda - i\tau(\hat \xi) +r & 0
\end{pmatrix},
$$
where $ \tau(\hat \xi)$ is an analytic function of $\xi$.
In this case, writing $\delta:= \hat \lambda - i\tau(\hat \xi) +r $, and making a similar computation, we
see that variations enter via a Puissieux series, through $\sqrt{\delta}$.
In particular, we find that, both in the inviscid and the viscous problem, the Evans function exhibits a square root-type
singularity in $\hat \lambda$ at $\hat \lambda=i \tau(\hat \xi)$.
This gives a useful check for multi-dimensional viscous computations; see \cite{HLyZ2}.

\subsection{Application in different frequency regimes}
By \eqref{asymptotics}, the balanced flux form of the Evans function is, with uniform stability, nonvanishing at the origin. This is useful for numerical conditioning in the delicate low-frequency regime.
To conveniently check intermediate frequencies by a robust
winding number computation, we may instead use the modified balanced flux formulation $D_{mb}$,
recovering the desirable property of analyticity 
in $\lambda$.
This reduces to the usual integrated Evans function for $\xi=0$,
and still has the desirable property that it is nonvanishing
at the origin (now multi-valued, depending on the limiting angle, as is the balanced flux version).
For an example of how this works in practice, 

%

\appendix
\section{Integrated coordinates: $b_{21}\ne 0$ ($d=1$)}\label{sec:b21n0}
In the case $d=1$, we indicate the changes incurred by dropping the condition $b^{jk}_{21}=0$ in \eqref{eq:bblock}. Suppose
\beq\label{eq:weak_b}
B^{11}(U)=
\bpm
0 & 0 \\ 
b_{21}(U) & b_{22}(U)
\epm\,.
\eeq
Linearizing about the steady solution $\bar U$, we 
obtain, as before, the linearized system
\beq \label{lin}
A^0(\bar U)U_t
+((A^1(\bar U) -sA^0(\bar U)) U)_x=(B^{11}(\bar U)U_x+\dif B^{11}(\bar U)(U,\bar U_x))_x\,,
\eeq
and we may write the associated eigenvalue equation as 
\beq\label{eq:eval}
\lambda \bar A^0U+(\bar A^1 U)'=(\bar B^{11} U')'\,.
\eeq
In \eqref{eq:eval}, we have written 
\be\label{coeffs}
\bar A^0:=A^0(\bar U)\,,
\quad \bar A^1U:=A^1(\bar U)U-sA^0(\bar U)-\dif B^{11}(\bar U)(U,\bar U_x)\,,\quad \bar B^{11}:=B^{11}(\bar U)\,.
\ee
Defining 
$w:=\bar A^0U$, $W'=w$, as in Section \ref{sec:integrate},
and integrating \eqref{eq:eval},
we have again \eqref{Weq}.
Setting now
$$
Z:=\bp  W\\ (b_{22}^{-1}b_{21},I_r) (\bar A^0)^{-1}W'\ep,
$$
we obtain by a similar, but more involved, computation to that in the case
$b_{21}=0$, $Z'=\mathbf{A}_\mathrm{int}Z$, where
\begin{equation}\label{eq:w_Aintb1}
\mathbf{A}_\mathrm{int} =
\bp
\bar A^0_{11}M_1+\bar A^0_{12}N_1 & 0 & \bar A^0_{11}M_3+\bar A^0_{12}N_3 \\
\bar A^0_{21}M_1+\bar A^0_{22}N_1 & 0 & \bar A^0_{21}M_3+\bar A^0_{22}N_3 \\
\mathbf{A}_{31} & \mathbf{A}_{32}  & \mathbf{A}_{33} 
\ep\,.
\end{equation}
Here, 
\begin{align}
\mathbf{A}_{31} & = b_2^{-1}\{-\lambda \bar A^1_{21}\mathcal{B}^{-1}-\lambda b_{22}(b_{22}^{-1}b_{21})'\mathcal{B}^{-1}+\lambda \bar A^1_{22}b_{22}^{-1}b_{21}\mathcal{B}^{-1}\}\,,\\
\mathbf{A}_{32}  & = b_{22}^{-1} \lambda \\
\mathbf{A}_{33} & = b_{22}^{-1}\{-\bar A^1_{21}\mathcal{B}^{-1}\bar A^1_{12}-b_{22}(b_{22}^{-1}b_{21})'\mathcal{B}^{-1}\bar A^1_{12}+\bar A^1_{22}+\bar A^1_{22}b_{22}^{-1}b_{21}\mathcal{B}^{-1}\bar A^1_{12}\}
\end{align}
and 
\begin{align}
M_1= -\lambda\mathcal{B}^{-1}\,,
M_3 = -\mathcal{B}^{-1}\bar A^1_{12} \,, 
N_1= \lambda b_{22}^{-1}b_{21}\mathcal{B}^{-1}\,,
N_3  = I+b_{22}^{-1}b_{21}\mathcal{B}^{-1}\bar A^1_{12}\,,
\end{align}
with
\begin{align}
\mathcal{B} & = \bar A^1_{11}-\bar A^1_{12}b_{22}^{-1}b_{21} \,.
\end{align}
Note that \eqref{eq:a11invalt} implies that $\mathcal{B}$ is invertible.
The system is closed by
solving for $(I,0)(\bar A^0)^{-1}W'$, 
whence, together with the coordinate 
$$
 (b_{22}^{-1}b_{21},I_r) (\bar A^0)^{-1}W',
$$
we obtain $(\bar A^0)^{-1}W'$ and thus $W'$.
For this step, multiply \eqref{Weq} by $(I,0)$ to obtain
$$
\lambda W_1 = - (\bar A^1_{11},\bar A^1_{12})(\bar A^0)^{-1}W'\,,
$$
from which we see that, provided the modified condition 
\eqref{eq:a11invalt} holds, 
we can solve for 
$(I,0)(\bar A^0)^{-1}W'$ 
in terms of the known coordinates
$W$ and 
$ (b_{22}^{-1}b_{21},I_r) (\bar A^0)^{-1}W'$.

\end{document}